\documentclass[12pt]{amsart}
\usepackage[T1]{fontenc}
\usepackage[latin9]{inputenc}
\usepackage{amscd}
\usepackage{amsthm}
\usepackage{amsmath}
\usepackage{amssymb}
\usepackage{esint}
\usepackage{mathrsfs}
\usepackage{a4wide}
\usepackage{color}
\usepackage{graphics}
\usepackage{amsrefs}

\newtheorem{Theorem}{Theorem}
\newtheorem{Lemma}[Theorem]{Lemma}

\newtheorem{Corollary}[Theorem]{Corollary}
\theoremstyle{remark}
\newtheorem{Remark}[Theorem]{Remark}
\theoremstyle{definition}

\begin{document}
\title{On the Fourier transforms of self-similar measures}
\author{Masato TSUJII}
\address{Department of Mathematics, Kyushu University, Fukuoka, 819-0395, JAPAN}
\email{tsujii@math.kyushu-u.ac.jp}
\date{\today}


\begin{abstract} For the Fourier transform $\mathcal{F}\mu$ of a general (non-trivial) self-similar measure~$\mu$ on the real line $\mathbb{R}$, we prove a large deviation estimate 
\[
\lim_{c\to +0} \varlimsup_{t\to \infty}\frac{1}{t}\log \left(\mathrm{Leb}\{x\in [-e^t , e^t]\mid |\mathcal{F}\mu(\xi)| \ge e^{-ct}  \}\right)=0.
\]
\end{abstract}
\maketitle

\def\real{\mathbb{R}}
\section{Introduction}
For  contractive affine maps $S_i:\real \to \real$, $i=1,2,\cdots, m$  and probabilities $0<p_i<1$ with $\sum_{i=1}^m p_i=1$, there exists a unique probability 
measure $\mu$ such that 
\[
\mu= \sum_{i=1}^m p_i \cdot \mu\circ S_i^{-1}.
\]
The probability measure $\mu$ thus defined is called {\em self-similar measure} and  has been studied extensively in fractal geometry. We refer \cite{MattilaBook} for the general background.    In this paper, we pose a few questions about the Fourier transforms of self-similar measures and then give a partial answer to one of them. 
To explain the questions, let us consider the  simple setting: 
\[
m=2,\quad S_1(x)=x/3, \quad S_2(x)=(x+2)/3, \quad p_1=p_2=1/2,
\]
in which the self-similar measure $\mu$ is the (normalized) Hausdoff measure of dimension 
\[
\delta=\log 2/\log 3
\]
on the standard middle-third Cantor set in $[0,1]$. It is not difficult to see that the Fourier transform of the measure $\mu$ is given as the infinite product
\begin{equation}\label{eq:product_formula}
\mathcal{F}\mu(\xi)=\prod_{\ell=1}^{\infty} \left(\frac{1+\exp(-2\cdot 3^{-\ell}\sqrt{-1}\cdot \xi)}{2}\right).
\end{equation}
A well-known result of Strichartz\cite{Strichartz}, which holds in more general setting,  gives the asymptotic formula
\[
\frac{1}{2R}\int_{-R}^{R} |\mathcal{F}\mu(\xi)|^2 d\xi =\mathcal{O}( R^{-\delta})\quad \mbox{ as $R\to \infty$}.
\]
This implies roughly that the square-value $|\mathcal{F}\mu(\xi)|^2$ decreases like $|\xi|^{-\delta}$ as $|\xi|$ tends to infinity provided that  it is averaged appropriately. 
We are interested in the deviation of the values of $\mathcal{F}\mu(\xi)$ from such  averaged behavior. 
  
Let us take logarithm of the both sides of (\ref{eq:product_formula}). Then we see
\begin{equation}
\log|\mathcal{F}\mu(\xi)|=\sum_{\ell=1}^\infty\psi(3^{-\ell} \xi)
\label{eq:Fmu}
\end{equation}
where
\[
\psi(x)=\log\frac{|1+e^{-2\sqrt{-1}x}|}{2}.
\]
Note that the sum on the right side of (\ref{eq:Fmu}) converges because, for each $\xi$,  the term $\psi(3^{-\ell} \xi)$ converges to $0$ exponentially fast as $\ell \to \infty$. 
By a slightly more precise consideration, we see that there exists  a constant $C>0$, independent of $N$, such that
\[
\left|\log|\mathcal{F}\mu(\xi)|-\sum_{\ell=1}^{N}\psi(3^{-\ell} \xi)\right|=
\left|\sum_{\ell=N+1}^\infty\psi(3^{-\ell} \xi)\right| <C 
\]
holds for all $\xi$ with $|\xi|\le 3^{N}\pi$. Therefore we may regard the function
\begin{equation}\label{psi}
\sum_{\ell=1}^{N}\psi(3^{-\ell} \xi)
\end{equation}
as an approximation of the function $\log|\mathcal{F}\mu(\xi)|$ on $[-3^N\pi,3^N\pi]$ with a bounded error term. 
By change of variable $\theta=3^{-N}\xi$, we see that  the (normalized) distribution of the values of the function (\ref{psi})  on $[-3^N\pi,3^N\pi]$ is same as that of the function 
\[
X_N:[-\pi, \pi]\to \real,\quad X_N(\theta)=\sum_{\ell=0}^{N-1}\psi(3^{\ell}\theta).
\]
The average of $X_N(\cdot)/N$ is
\begin{align*}
A:&=\frac{1}{N\cdot 2\pi }\int_{-\pi}^{ \pi} 
X_N(\theta) d\theta=\frac{1}{2\pi}\int_{-\pi}^\pi \psi(\xi) d\xi. \end{align*}
To analyze the deviation of  the values of $X_N(\cdot)/N$ from the average  $A$, 
it is useful to introduce the view point of dynamical system. 
Let us consider the dynamical system generated by the map
\[
T:\real/2\pi \mathbb{Z}\to \real/2\pi \mathbb{Z}, \qquad T(x)=3x \mod 2\pi\mathbb{Z}
\]
and regard the function $X_N(\theta)/N$  as the Birkhoff average of the observable $\psi$  along the orbit of $\theta$:
\[
\frac{1}{N}X_N(\theta)=\frac{1}{N}\sum_{\ell=0}^{N-1}\psi\circ T^\ell(\theta).
\]  
Then the results of Takahashi\cite{Takahashi87} and Young\cite{Young90} on the large deviation principle in dynamical system setting tells 
 that\footnote{Actually we can not apply the results in \cite{Takahashi87} and \cite{Young90} directly because the function $\psi$ has logarithmic singularities at $x=\pm \pi/2$ and hence is not continuous. Still we can obtain the large deviation estimate stated here by approximating $\psi$ by smooth functions from above. It seems reasonable to expect that the limit is exact and the inequality is actually an equality.}, for any $c\ge 0$, it holds
\[
\varlimsup_{N\to \infty} \frac{1}{N}\log \left(\frac{m\{x\in [-3^N \pi, 3^N\pi]\mid |\mathcal{F}\mu(\xi)| \ge e^{-cN}  \}}{3^N\cdot 2\pi}  \right)\le \widehat{R}(c)-\log 3
\]
with setting
\begin{equation}\label{eq:Rc}
\widehat{R}(c)=\sup\left \{h_{\nu}(T)\;\left|\; \nu\in \mathcal{M}_T \mbox{ such that } \int \psi d\nu\ge -c\right.\right\}\ge 0
\end{equation}
where $\mathcal{M}_T$ denotes the space of $T$-invariant Borel probability measures and $h_\nu(T)$  the measure theoretical entropy of $T$ with respect to a  measure $\nu\in \mathcal{M}_T$. Since the  entropy $h_{\nu}(T)$ is an upper semi-continuous functional on the space $\mathcal{M}_T$ with respect to the weak topology, we may take the supremum in (\ref{eq:Rc}) as the maximum.

The function $\widehat{R}(c)$ is  increasing with respect to $c$  by definition. Further we can make the following observations: (Figure \ref{fig})
\begin{itemize}
\item[(A)] $\widehat{R}(c)-\log 3\le 0$ and the inequality is strict if and only if 
$-c>A$ (or $c<|A|$), because  the (normalized) uniform measure is the unique maximal entropy ($=\log 3$) measure for~$T$.
\item[(B)] $\lim_{c\to +0} \widehat{R}(c)=0$. This is because, 
if $c>0$ is close to $0$, a $T$-invariant probability measure $\nu$ satisfying $(2\pi)^{-1}\int \psi d\nu\ge -c$ must concentrate mostly on a small neighborhood of the points $0, \pi\in \real/2\pi \mathbf{Z}$ at which $\psi$ attains its maximum value $0$.
\item[(C)] The function $\widehat{R}(c)$ is a convex function on $[0,|A|]$. This follows  from the definition of $\widehat{R}(c)$ and the affine property of the entropy $h_{\nu}(T)$ with respect to $\nu$.
\end{itemize}
\begin{figure}
\begin{center}
%WinTpicVersion4.22
\unitlength 0.1in
\begin{picture}( 30.1000, 22.7500)(  4.0000,-28.1000)
% VECTOR 2 0 3 0 Black White
% 4 780 2810 780 630 580 2410 3680 2410
% 
{\color[named]{Black}{%
\special{pn 8}%
\special{pa 780 2810}%
\special{pa 780 630}%
\special{fp}%
\special{sh 1}%
\special{pa 780 630}%
\special{pa 760 698}%
\special{pa 780 684}%
\special{pa 800 698}%
\special{pa 780 630}%
\special{fp}%
\special{pa 580 2410}%
\special{pa 3680 2410}%
\special{fp}%
\special{sh 1}%
\special{pa 3680 2410}%
\special{pa 3614 2390}%
\special{pa 3628 2410}%
\special{pa 3614 2430}%
\special{pa 3680 2410}%
\special{fp}%
}}%
% LINE 2 2 3 0 Black White
% 4 790 990 2580 990 2580 990 2580 2390
% 
{\color[named]{Black}{%
\special{pn 8}%
\special{pa 790 990}%
\special{pa 2580 990}%
\special{dt 0.045}%
\special{pa 2580 990}%
\special{pa 2580 2390}%
\special{dt 0.045}%
}}%
% SPLINE 1 0 3 0 Black White
% 4 780 2400 1030 1810 2570 990 2570 990
% 
{\color[named]{Black}{%
\special{pn 13}%
\special{pa 780 2400}%
\special{pa 790 2370}%
\special{pa 802 2338}%
\special{pa 812 2306}%
\special{pa 822 2276}%
\special{pa 832 2244}%
\special{pa 844 2214}%
\special{pa 854 2182}%
\special{pa 926 2002}%
\special{pa 982 1890}%
\special{pa 1030 1812}%
\special{pa 1048 1786}%
\special{pa 1064 1762}%
\special{pa 1084 1736}%
\special{pa 1102 1714}%
\special{pa 1122 1690}%
\special{pa 1162 1646}%
\special{pa 1184 1624}%
\special{pa 1204 1602}%
\special{pa 1226 1582}%
\special{pa 1250 1562}%
\special{pa 1272 1542}%
\special{pa 1296 1524}%
\special{pa 1320 1504}%
\special{pa 1346 1486}%
\special{pa 1370 1468}%
\special{pa 1396 1452}%
\special{pa 1422 1434}%
\special{pa 1448 1418}%
\special{pa 1476 1402}%
\special{pa 1502 1386}%
\special{pa 1558 1354}%
\special{pa 1586 1340}%
\special{pa 1616 1324}%
\special{pa 1644 1310}%
\special{pa 1704 1282}%
\special{pa 1734 1270}%
\special{pa 1764 1256}%
\special{pa 1796 1244}%
\special{pa 1826 1232}%
\special{pa 1858 1218}%
\special{pa 1888 1206}%
\special{pa 1952 1182}%
\special{pa 1984 1172}%
\special{pa 2016 1160}%
\special{pa 2050 1148}%
\special{pa 2114 1128}%
\special{pa 2148 1116}%
\special{pa 2180 1106}%
\special{pa 2282 1076}%
\special{pa 2314 1064}%
\special{pa 2382 1044}%
\special{pa 2416 1036}%
\special{pa 2552 996}%
\special{pa 2570 990}%
\special{fp}%
}}%
% STR 2 0 3 0 Black White
% 4 400 875 400 925 2 0 0 0
% Koko
\put(4.0000,-10.2500){\makebox(0,0)[lb]{$\log 3$}}%
% STR 2 0 3 0 Black White
% 4 2520 2465 2520 2515 2 0 0 0
% Ko
\put(25.2000,-26.1500){\makebox(0,0)[lb]{$|A|$}}%
% LINE 1 0 3 0 Black White
% 2 2580 990 3165 990
% 
{\color[named]{Black}{%
\special{pn 13}%
\special{pa 2580 990}%
\special{pa 3166 990}%
\special{fp}%
}}%
% STR 2 0 3 0 Black White
% 4 3410 2425 3410 2475 2 0 0 0
% x
\put(34.1000,-25.7500){\makebox(0,0)[lb]{c}}%
% STR 2 0 3 0 Black White
% 4 680 630 680 680 2 0 0 0
% y
\put(6.8000,-5.8000){\makebox(0,0)[lb]{$\widehat{R}(c)$}}%
\end{picture}%
\end{center}
\caption{A schematic picture of the graph of $\widehat{R}(c)$}
\label{fig}
\end{figure}
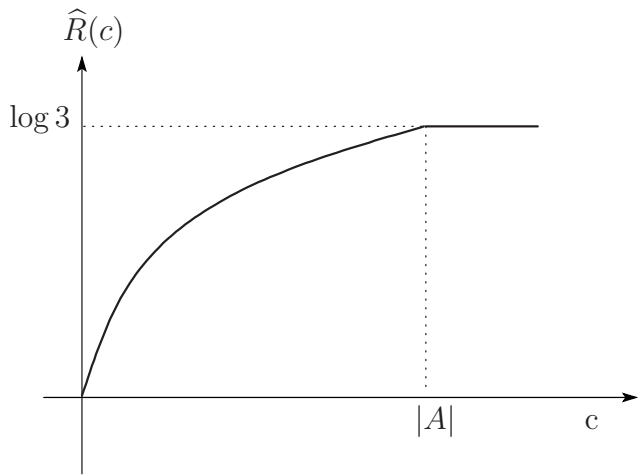
The question that we would like to pose is whether and/or how these observations remain true for more general classes of self-similar measures (or classes of dynamically defined measures). 
In this paper, we give a simple (and rather modest) result which generalizes the observation (B) to general (non-trivial) self-similar measures.

\section{Result} 
Let $m>1$ be an integer. 
For  $i=1,2,\cdots, m$, let 
\[
S_i:\real\to \real, \quad S_i(x)=a_i x+b_i
\]
be contacting affine maps, {\it i.e. } $|a_i|<1$.
Let $0<p_i<1$, $i=1,2,\cdots, m$, be real numbers such that $\sum_{i=1}^m p_i=1$. Then 
 there exists a unique probability 
measure $\mu$ such that 
\[
\mu= \sum_{i=1}^m p_i \cdot \mu\circ S_i^{-1}.
\]
The support $K$ of the measure $\mu$  is the unique compact subset that satisfies 
\[
K=\bigcup_{i=1}^m S_i(K).
\]
We henceforth assume that the subset $K$ is not a single point, to avoid the trivial case.
The Fourier transform  $\mathcal{F}\mu$ of the measure $\mu$ is  the real-analytic function defined by
\[
\mathcal{F} \mu(\xi)=\int \exp(-\sqrt{-1}\xi x)\, d\mu(x).
\]
Let $R(c\,;\mu)$ be the function defined for $\mu$ by
\begin{equation}\label{eq:Rcm}
R(c\,;\mu)=
\varlimsup_{t\to \infty} \frac{1}{t}\log \left(\mathrm{Leb}\{x\in [-e^t , e^t]\mid |\mathcal{F}\mu(\xi)| \ge e^{-ct}  \}\right).
\end{equation}
Note that this coincides with the quantity $\widehat{R}(c)$ in the special case considered in the last section. 
The following is the main result of this paper, which generalizes the observation (B) in the last section to general (non-trivial) self-similar measures. 
\begin{Theorem}\label{th:main}  $\lim_{c\to +0} R(c\,;\mu)=0$. 
\end{Theorem}

\section{Proof}

\subsection{Preliminaries}
Let $\mathcal{A}=\{1,2,\cdots, m\}$. Let $\mathbf{p}$ be the probability measure on $\mathcal{A}$ such that  $\mathbf{p}(\{i\})=p_i$.  
We denote by  $\mathcal{A}^n$ the set of words of length $n$ with letters in $\mathcal{A}$ and by  $\mathbf{p}^n$  the probability measure on it obtained as  the $n$-times direct product of $\mathbf{p}$.

We regard the system $\{(S_i, p_i)\}_{i\in \mathcal{A}}$ as a random dynamical system in which the mapping $S_i$ is applied with probability $p_i$. 
Then the measure $\mu$ in Theorem \ref{th:main} is the unique stationary measure for it.  
The $n$-th iteration of the random dynamical system $\{(S_i, p_i)\}_{i\in \mathcal{A}}$ is  the system $\{(S_{\mathbf{i}}, p_{\mathbf{i}})\}_{\mathbf{i}\in \mathcal{A}^n}$ where
\[
S_{\mathbf{i}}:\real\to \real, \quad S_{\mathbf{i}}=S_{i(n)}\circ S_{i(n-1)}\circ \cdots \circ S_{i(1)}
\]
and
\[
p_{\mathbf{i}}:=\mathbf{p}^n(\mathbf{i})=p_{i(1)}\cdot p_{i(2)}\cdots p_{i(n)}
\]
for $\mathbf{i}=(\mathbf{i}(n), \mathbf{i}(n-1), \cdots, \mathbf{i}(2), \mathbf{i}(1))\in \mathcal{A}^n$. 
The affine map $S_{\mathbf{i}}$ is expressed as
\[
S_{\mathbf{i}}(x)=a_{\mathbf{i}} \cdot x+b_{\mathbf{i}},
\]
and we have
\[
a_{\mathbf{i}}=\prod_{k=1}^{n}a_{\mathbf{i}(k)}.
\]
\begin{Remark}\label{rm:order}
We will order the element of  the sequence $\mathbf{i}\in \mathcal{A}^n$ from the right to the left:
\[
 \mathbf{i}=(\mathbf{i}(n), \mathbf{i}(n-1), \cdots, \mathbf{i}(2), \mathbf{i}(1)).
 \]
 This notation is more natural in our argument, though it is not very essential. 
\end{Remark}
Note that the measure $\mu$ is the unique invariant measure also for the iterations $\{(S_{\mathbf{i}}, p_{\mathbf{i}})\}_{\mathbf{i}\in \mathcal{A}^n}$ for $n\ge 1$. Below we develop our argument for the system $\{(S_i, p_i)\}_{i\in \mathcal{A}}$.  
But our argument is applicable to the iterations of $\{(S_i, p_i)\}_{i\in \mathcal{A}}$ in parallel and, in a few places, we will replace the system $\{(S_i, p_i)\}_{i\in \mathcal{A}}$ by its iterate in order to assume some numerical conditions. 
For instance, we may (and will) assume  
\begin{equation}\label{eq:a}
|a_i|<\frac{1}{2} \quad \mbox{for all $1\le i\le m$}
\end{equation}
without loss of generality, by such replacement. 

Any changes of coordinate on $\real$ by affine bijections do not affect validity of  the main theorem. Therefore we may and do  assume also
\begin{equation}
\label{eq:b}b_0=0\le b_1\le \cdots \le b_m=1.
\end{equation}
It is then easy to see that the invariant subset $K$ is contained in $[-2,2]$.

\subsection{The average rate of contraction}
The Lyapunov exponent of the random dynamical system $\{S_i, p_i\}_{i=1}^m$ is
\[
\chi:=\sum_{i=1}^m p_i\cdot  \log |a_i|<0.
\]
The following is a simple application of the large deviation principle\cite{Varadhan}.
\begin{Lemma}\label{lm1}For any $\delta>0$, there exists  $\epsilon>0$  such that 
\[
\mathbf{p}^n\{ \mathbf{i}\in \mathcal{A}^n\mid |a_{\mathbf{i}}| \le \exp((\chi-\delta)n) \} \le  e^{-\epsilon n}
\]
and
\[
\mathbf{p}^n\{ \mathbf{i}\in \mathcal{A}^n\mid |a_{\mathbf{i}}| \ge \exp((\chi+\delta)n)\} \le  e^{-\epsilon n}
\]
hold for sufficiently large  $n\ge 0$.
\end{Lemma}
Take and fix integers $0<\check{r}<\hat{r}$ with $\hat{r}-\check{r}\le 2$ so that 
\[
\check{r}<e^{|\chi|}<\hat{r}.
\]
As an immediate consequence of the last lemma, we get
\begin{Corollary}\label{cor}There exists a constant $\epsilon>0$ such that 
\[
\mathbf{p}^n\{ \mathbf{i}\in \mathcal{A}^n\mid |a_{\mathbf{i}}| \ge \check{r}^{-n}\mbox{ or } |a_{\mathbf{i}}| \le \hat{r}^{-n} \} \le  e^{-\epsilon n}
\]
for sufficiently large  $n\ge 0$.
\end{Corollary}
Notice that, replacing $\{S_i, p_i\}_{i=1}^m$ by its iteration, we may assume that the integers $\hat{r}$ and $\check{r}$ are arbitrarily large and that the rate $\hat{r}/\check{r}$ is arbitrarily close to $1$.

\subsection{The sequence of measures $\mu_n$ and their Fourier transform}
Let $\delta_0$ be the Dirac measure at the origin $0\in \real$ and set 
\[
\mu_n=\sum_{\mathbf{i}\in \mathcal{A}^n} p_{\mathbf{i}}\cdot \delta_0\circ S_\mathbf{i}^{-1}
\quad \mbox{for $n\ge 0$.}
\]
The measure $\mu_n$ converges to $\mu$ as $n\to\infty$ in the weak topology and hence 
its Fourier transform $\mathcal{F}\mu_n(\xi)$ converges to $\mathcal{F}\mu (\xi)$ for each  $\xi\in \real$. The convergence is actually uniform on any compact subset. Further we have
 \begin{Lemma}\label{lm:apx} There exists a constant $\epsilon>0$ such that, for sufficiently large $n$, we have
\[
|\mathcal{F}\mu (\xi) -\mathcal{F}\mu_n(\xi)|\le e^{-\epsilon n}\quad \mbox{uniformly for $\xi\in \real^d$ with $|\xi|\le \check{r}^n$.}
\] 
\end{Lemma}
\begin{proof} From invariance of $\mu$, we have 
\begin{equation}\label{eq:Fm}
\mathcal{F}\mu (\xi) -\mathcal{F}\mu_n(\xi)=
\sum_{\mathbf{i}\in \mathcal{A}^n}p_{\mathbf{i}}\cdot 
\int e^{-\sqrt{-1} \xi \cdot x} d\left((\mu-\delta_0) \circ S_{\mathbf{i}}^{-1}\right)(x).
\end{equation}
Take a real number $r$ in between $\check{r}$ and $e^{|\chi|}$. From Lemma \ref{lm1}, we see that 
\[
\mathbf{p}^n\{ \mathbf{i}\in \mathcal{A}^n\mid |a_{\mathbf{i}}| \ge r^{-n} \} 
\]
is exponentially small. That is, the sum in (\ref{eq:Fm})  restricted to $\mathbf{i}\in \mathcal{A}^n$ with $|a_{\mathbf{i}}|\ge r^{-n}$ is exponentially small with respect to $n$. 
If $|a_{\mathbf{i}}|<r^{-n}$, we have that 
\[
\left|\int e^{-\sqrt{-1}\xi \cdot x} d\left((\delta_0-\mu) \circ S_{\mathbf{i}}^{-1}\right)(x)\right|< |\xi| \cdot r^{-n}\cdot \mathrm{diam} K< 4(\check{r}/r)^n
\]
provided $|\xi|\le \check{r}^n$.  Therefore the remaining part of the sum  (\ref{eq:Fm}) is also exponentially small with respect to $n$. 
\end{proof}

To continue, we introduce the operators 
\[
T_i:C^\infty(\real)\to C^\infty(\real), \quad T_i u(\xi)=e^{-\sqrt{-1}b_i\xi} \cdot u(a_{i}\cdot \xi)
\]
defined for $i\in \mathcal{A}$. 
Since  $
\mathcal{F}(\nu\circ S_i^{-1})=T_i (\mathcal{F} \nu)$, 
we have
\begin{equation}\label{eq:relation_Fmu_k}
\mathcal{F}\mu_{k+1}=\sum_{i\in \mathcal{A}} p_i\cdot T_i (\mathcal{F}\mu_{k}).
\end{equation}
Recursive application of this relation yields
\[
\mathcal{F}\mu_n=\sum_{\mathbf{i}\in \mathcal{A}^n}p_{\mathbf{i}}\cdot T_{\mathbf{i}}(\mathcal{F}\delta_0)=\sum_{\mathbf{i}\in \mathcal{A}^n}p_{\mathbf{i}}\cdot T_{\mathbf{i}}(\mathbf{1})
\]
where $
T_{\mathbf{i}} =T_{\mathbf{i}(n)}\circ T_{\mathbf{i}(n-1)}\circ \cdots \circ T_{\mathbf{i}(1)}$. 

Let $n$ be a large positive integer. Notice that below we drop the dependence on $n$ from the notation for simplicity, though most of our constructions depend on $n$.
For  $\mathbf{i}\in \mathcal{A}^k$ with $0\le k\le n$, let  $
X_{\mathbf{i}}=X_{\mathbf{i}}^{(n)}:I(\mathbf{i})\to \mathbb{C}$
be the restriction of the function $\mathcal{F}\mu_{n-k}(\xi)$ to the interval 
\[
I(\mathbf{i}):=[-|a_{\mathbf{i}}|\cdot \check{r}^n, |a_\mathbf{i}|\cdot \check{r}^n].
\]
For convenience, we suppose that $\mathcal{A}^0$ as the set that consists of the single element $\emptyset$ and that $a_\emptyset=1$.
\begin{Remark}
The function $X_{\mathbf{i}}(\xi)$ and the interval $I(\mathbf{i})$ actually depend only on the length $k=|\mathbf{i}|$ of the sequence $\mathbf{i}$ and the number $a_{\mathbf{i}}$. 
\end{Remark}

From the definition, we have that 
\begin{equation}\label{eq:init}
X_{\mathbf{i}}(\xi)=\mathcal{F}\delta_0(\xi)\equiv 1\quad \mbox{for all $\mathbf{i}\in \mathcal{A}^n$}
\end{equation}
and also that
\[
X_{\emptyset}(\xi)=\mathcal{F}\mu_{n}(\xi). 
\]
For two sequences $\mathbf{i}\in \mathcal{A}^k$ and $\mathbf{j}\in \mathcal{A}^{k'}$, let $\mathbf{i}\cdot \mathbf{j}\in \mathcal{A}^{k+k'}$ be the concatenation of them:
\[
\mathbf{i}\cdot \mathbf{j}(\ell)=\begin{cases} \mathbf{j}(\ell),&\quad \mbox{if $1\le \ell\le k'$;}\\
\mathbf{i}(\ell-k'), &\quad \mbox{if $k'+1\le \ell\le k'+k$.}
\end{cases}
\]
(Recall Remark \ref{rm:order}.)
Then, from (\ref{eq:relation_Fmu_k}), the functions $X_{\mathbf{i}}(\xi)$ for $\mathbf{i}\in \mathcal{A}^k$ with $0\le k\le n-1$ satisfy\begin{equation}\label{eq:relX}
X_{\mathbf{i}}(\xi)=\sum_{j\in \mathcal{A}}p_j\cdot T_{j}X_{\mathbf{i}\cdot j}(\xi)
\end{equation}
or more generally 
\begin{equation}\label{eq:relX2}
X_{\mathbf{i}}(\xi)=\sum_{\mathbf{j}\in \mathcal{A}^{k'}}p_{\mathbf{j}}\cdot T_{\mathbf{j}}X_{\mathbf{i}\cdot \mathbf{j}}(\xi)
\end{equation}
for $1\le k'\le n-k$.
\subsection{The auxiliary quantity $Y_{\mathbf{i}}(\cdot)$}
Our task is to analyze how the summands on the right hand side of (\ref{eq:relX}) cancel each other by the difference of the complex phase. One technical problem in such analysis is that, if a few of the summands are much larger in absolute value than the others  and they have coherent complex phases, the cancellation will not be effective. In order to deal with such problem, we introduce another  family of positive real-valued functions 
\[
Y_{\mathbf{i}}:I(\mathbf{i})\to \mathbb{R}_+
\]
for $\mathbf{i}\in \mathcal{A}^k$ with $0\le k\le n$. 
For $\mathbf{i}\in \mathcal{A}^n$, we set 
\begin{equation}\label{eq:YXn}
Y_{\mathbf{i}}(\xi)=X_{\mathbf{i}}(\xi)\equiv 1\quad \mbox{on $I(\mathbf{i})$.}
\end{equation}
Then we define $Y_{\mathbf{i}}(\xi)$ for $\mathbf{i}\in \mathcal{A}^k$ inductively (in descending order in $k$) by the relation 
\[
Y_{\mathbf{i}}(\xi)=\max\left\{\,\frac{1}{2}\sum_{j=1}^{m}p_j\cdot Y_{\mathbf{i}\cdot j}(a_j\cdot \xi),\; |X_{\mathbf{i}}(\xi)| \,\right\}.
\]
From this definition and the relation (\ref{eq:relX}), we see that 
\[
 |X_{\mathbf{i}}(\xi)|\le Y_{\mathbf{i}}(\xi)\le 1
\]
and that
\begin{equation}\label{eq:relY}
\frac{1}{2}\sum_{j\in \mathcal{A}}p_j\cdot Y_{\mathbf{i}\cdot j}(a_j\cdot \xi)\le Y_{\mathbf{i}}(\xi)
\le 
\sum_{j\in \mathcal{A}}p_j\cdot Y_{\mathbf{i}\cdot j}(a_j\cdot \xi).
\end{equation}
The latter inequality gives by induction that
\begin{equation}\label{eq:relY2}
\frac{1}{2^\ell}\sum_{\mathbf{j}\in \mathcal{A}^\ell}p_{\mathbf{j}}\cdot Y_{\mathbf{i}\cdot \mathbf{j}}(a_{\mathbf{j}}\cdot \xi)\le Y_{\mathbf{i}}(\xi)
\le 
\sum_{\mathbf{j}\in \mathcal{A}^\ell} p_{\mathbf{j}}\cdot Y_{\mathbf{i}\cdot \mathbf{j}}(a_{\mathbf{j}}\cdot \xi)
\end{equation}
for $\mathbf{i}\in \mathcal{A}^{k}$ with $0\le k\le n-1$ and for $1\le \ell\le n-k$.

The next lemma is the main reason to introduce  the functions $Y_{\mathbf{i}}(\cdot)$. (See the remark below.)
\begin{Lemma}\label{lm:Lip} For any $\mathbf{i}\in \mathcal{A}^k$ with $0\le k\le n$, we have
\[
\left|\frac{d}{d\xi}X_{\mathbf{i}}(\xi)\right| \le \frac{2}{1-(2\max_{i\in \mathcal{A}}|a_i|)}\cdot  Y_{\mathbf{i}}(\xi)\quad \mbox{for any $\xi\in  I(\mathbf{i})$}
\]
and
\[
\left|\frac{d}{d\xi}Y_{\mathbf{i}}(\xi)\right| \le \frac{2}{1-(2\max_{i\in \mathcal{A}}|a_i|)}\cdot  Y_{\mathbf{i}}(\xi)\quad \mbox{for Lebesgue almost all  $\xi\in  I(\mathbf{i})$}
\]
where the differential  in the second inequality is considered in the sense of distribution. (Note that $Y_{\mathbf{i}}$ satisfies the Lipschitz condition and hence $\frac{d}{d\xi}Y_{\mathbf{i}}\in L^\infty$.)
\end{Lemma}
\begin{Remark} The second inequality implies that the functions $Y_{\mathbf{i}}$ are tame in the sense that  the differentials of $\log Y_{\mathbf{i}}$ are bounded uniformly. The first inequality implies that, if $|X_{\mathbf{i}}|$ is comparable with $Y_{\mathbf{i}}$,  the differential of $\log X_{\mathbf{i}}$ is also bounded. The case where $|X_{\mathbf{i}}|$ is much smaller than $Y_{\mathbf{i}}$ will be treated in a different manner. 
\end{Remark}
\begin{proof}Applying  the chain rule to the inductive relation (\ref{eq:relX}) and using the assumption $|b_i|\le 1$, we see
\begin{align*}
\left|\frac{d}{d\xi}X_{\mathbf{i}}(\xi)\right|&\le \sum_{\ell=1}^{n-k} \sum_{\mathbf{j}\in \mathcal{A}^{\ell}} p_{\mathbf{j}}\cdot  \frac{ |a_{\mathbf{j}}|\cdot |b_{\mathbf{j}(1)}|}{|a_{\mathbf{j}(1)}|} \cdot |X_{\mathbf{i}\cdot \mathbf{j}}(a_{\mathbf{j}}\cdot \xi)|\\
&\le 
\sum_{\ell=1}^{n-k} \sum_{\mathbf{j}\in \mathcal{A}^{\ell}} p_{\mathbf{j}}\cdot 
\frac{ |a_{\mathbf{j}}|}{|a_{\mathbf{j}(1)}|} \cdot Y_{\mathbf{i}\cdot \mathbf{j}}(a_{\mathbf{j}}\cdot \xi)\\
&\le 
\sum_{\ell=1}^{n-k}2\cdot \left(2\max_{i\in \mathcal{A}}|a_i|\right)^{\ell-1}
 \cdot Y_{\mathbf{i}}(\xi)\quad \mbox{by (\ref{eq:relY2})}\\
 &\le \frac{2}{1-(2\max_{i\in \mathcal{A}}|a_i|)}\cdot Y_{\mathbf{i}}(\xi).
\end{align*}
We can get the second claim by a similar inductive argument on the length of $\mathbf{i}$, using the first claim. 
\end{proof}

\subsection{The main estimate}

For $\mathbf{i}\in \mathcal{A}^k$ with $0\le k\le n$,  
let $\Xi_{\mathbf{i}}$ be the partition of the interval $[-\check{r}^n, \check{r}^n)$ defined as follows: For $\mathbf{i}=\emptyset\in \mathcal{A}^0$, let $\Xi_{\emptyset}$ be the partition of the interval $I(\emptyset)=[-\check{r}^n, \check{r}^n)$ into the unit  intervals
\[
[k, k+1)\quad \quad \mbox{for $-\check{r}^n \le k< \check{r}^n$.}
\] 
Then we construct the partition $\Xi_{\mathbf{i}}$ for 
$\mathbf{i}\in \mathcal{A}^k$ so that 
each element of $\Xi_{\mathbf{i}\cdot j}$ is a union of consecutive intervals in the partition $\Xi_{\mathbf{i}}$ and that 
\[
\min\{|a_{\mathbf{i}}|^{-1}, 2\check{r}^n\}\le |I|\le 2|a_{\mathbf{i}}|^{-1}\quad \mbox{for $I\in \Xi_{\mathbf{i}}$.}
\]
 Such partition is not unique of course. But any of such partitions will work in the following argument. Note that, when $|a_{\mathbf{i}}|^{-1}\ge 2\check{r}^n$, the partition $\Xi_{\mathbf{i}}$ consists of the single interval $[-\check{r}^n, \check{r}^n)$ and, otherwise, we have
 \[
|a_{\mathbf{i}}|^{-1}\le |I|\le 2|a_{\mathbf{i}}|^{-1}\quad \mbox{for $I\in \Xi_{\mathbf{i}}$.}
 \]

The next lemma deals with the technical part of the proof of Theorem \ref{th:main}. 
\begin{Lemma}\label{lm:main}
There exists $\eta>0$, independent of $n$,  such that the following holds true:
Consider arbitrary $j\in \mathcal{A}$,  $\mathbf{i}\in \mathcal{A}^k$ with $0\le k\le n-1$ and an interval  $I\in \Xi_{\mathbf{i}\cdot j}$.  Suppose that 
the interval $I$ is the union of consecuetive intervals  $I_1, I_2, \cdots, I_N\in \Xi_{\mathbf{i}}$. Set
\begin{equation}\label{eq:Id}
I'_\nu:=a_{\mathbf{i}\cdot j}\cdot I_\nu\subset I':=a_{\mathbf{i}\cdot j}\cdot I\quad \mbox{ for $1\le \nu\le N$.}
\end{equation}
Then we have
\begin{equation}\label{eq:mainest}
Y_{\mathbf{i}}(\xi)^2
\le \exp(-\eta) \cdot \sum_{j\in \mathcal{A}} p_j\cdot Y_{\mathbf{i}\cdot j}(a_j\cdot \xi)^2\quad \mbox{for $\xi \in I'_\nu$}
\end{equation}
for all $1\le \nu \le N$ but for at most two exceptions. 
\end{Lemma}
\begin{proof} We may and do assume $N>2$ because the claim is vacuous otherwise. 
This implies  $|a_{\mathbf{i}}|^{-1}< 2\check{r}^n$ as we noted above. In particular, we have 
\[
|I|\le 2 |a_{\mathbf{i}\cdot j}|^{-1} \quad\mbox{and}\quad  |a_{\mathbf{i}}|^{-1}\le |I_\nu|\le 2|a_{\mathbf{i}}|^{-1}
\]
and hence
\begin{equation}\label{eq:lengthI}
|I'|\le 2\quad\mbox{and}\quad |a_{j}|\le |I'_\nu|\le 2|a_{j}|\le 1.
\end{equation}
We first show that the claim of the lemma can be proved by simple manners if either of the following two conditions holds:
\begin{itemize}
\item[(I)] 
$Y_{\mathbf{i}\cdot j}(a_j\cdot \xi)\le Y_{\mathbf{i}\cdot j'}(a_{j'}\cdot \xi)/2$ on $I'$ 
for some pair  $(j,j')\in \mathcal{A}\times \mathcal{A}$. 
\item[(II)]  $X_{\mathbf{i}\cdot j_0}(a_{j_0}\cdot\xi)\le Y_{\mathbf{i}\cdot j_0}(a_{j_0}\cdot\xi)/2$ on $I'$ for some $ j_0\in \mathcal{A}$.
\end{itemize}

Suppose that  the condition (I) holds. 
Setting $a_j=\sqrt{p_j}$ and $b_{j}=\sqrt{p_j}\cdot Y_{\mathbf{i}\cdot j}(a_{j}\cdot \xi)$ in the identity
\[
\left(\sum_{j\in \mathcal{A}} a_j b_j\right)^2=\left(\sum_{j\in \mathcal{A}} a_j^2\right)\left(\sum_{j\in \mathcal{A}} b_j^2\right)-\frac{1}{2} \sum_{j,j'\in \mathcal{A}}(a_{j}b_{j'}-a_{j'}b_{j})^2,
\]
we get

\begin{align*}
\left(\sum_{j\in \mathcal{A}} p_j \cdot Y_{\mathbf{i}\cdot j}(a_{j}\cdot \xi)\right)^2&= \sum_{j\in \mathcal{A}} p_j \cdot Y_{j\cdot\mathbf{i}}(a_{j}\cdot \xi)^2
 - \sum_{j, j'\in \mathcal{A}} p_{j}p_{j'}\cdot 
\frac{|Y_{\mathbf{i}\cdot j}(a_{j}\cdot \xi)-Y_{\mathbf{i}\cdot j'}(a_{j'}\cdot \xi)|^2}{2}.
\end{align*}
Hence we have, from (\ref{eq:relY}), that
\begin{align*}
1-\frac{Y_{\mathbf{i}}(\xi)^2}{\sum_{j\in \mathcal{A}} p_j\cdot Y_{\mathbf{i}\cdot j}(a_j\cdot \xi)^2}
&\ge 1-
\frac{\left(\sum_{j\in \mathcal{A} } p_j\cdot Y_{\mathbf{i}\cdot j}(a_{j}\cdot \xi)\right)^2}
{\sum_{j\in \mathcal{A}} p_j\cdot Y_{\mathbf{i}\cdot j }(a_j\cdot \xi)^2}\\
&\ge
\frac{p_{\min}^2\cdot |Y_{\max}(\xi)-Y_{\min}(\xi)|^2 }
{2\cdot \sum_{j\in \mathcal{A}} p_j\cdot Y_{\mathbf{i}\cdot j }(a_j\cdot \xi)^2}\\
&\ge
\frac{p_{\min}^2\cdot |Y_{\max}(\xi)-Y_{\min}(\xi)|^2 }
{2\cdot Y_{\max}(\xi)^2}
\end{align*}
with setting $p_{\min}=\min_{i\in \mathcal{A}} p_i>0$ and
\[
Y_{\max}(\xi)=\max_{j\in \mathcal{A}}Y_{\mathbf{i}\cdot j}(a_j\cdot \xi),\qquad 
Y_{\min}(\xi)=\min_{j\in \mathcal{A}}Y_{\mathbf{i}\cdot j}(a_j\cdot \xi).
\]
From the condition  (I), we have $Y_{\min}(\xi)\le Y_{\max}(\xi)/2$ on $I'$ and hence  the right hand side of the inequality above is not less than $
p_{\min}^2/8$ for $\xi\in I'$. 
Therefore, choosing $\eta>0$ so small that  $e^{-\eta}>1-p_{\min}^2/8$, we obtain the conclusion of the lemma (with no exception for $1\le \nu\le N$). 

Next suppose that  the condition (II) holds and that the condition (I) does {\em not} hold.  
From the latter condition,   there exists a point $\bar{\xi}=\bar{\xi}(j,j') \in I'$ for each pair $(j, j')\in \mathcal{A}\times \mathcal{A}$ such that 
\begin{equation}\label{eq:Y}
Y_{ \mathbf{i}\cdot j}(a_j\cdot \bar{\xi})> Y_{\mathbf{i}\cdot j'}(a_{j'}\cdot \bar{\xi})/2.
\end{equation}
From  Lemma \ref{lm:Lip} and (\ref{eq:lengthI}), we have
\begin{equation}\label{eq:Ydis}
\frac{Y_{\mathbf{i}\cdot j}(a_j\cdot \xi)}{Y_{\mathbf{i}\cdot j}(a_j\cdot \xi')}
\le 
\alpha:=\exp\left(\frac{4\cdot \max_{1\le i\le m}|a_i|}{1-2(\max_{1\le i\le m}|a_i|)}\right)
\end{equation}
for all $\xi, \xi'\in I'$ and $j\in \mathcal{A}$. Hence  
\begin{equation}\label{eq:y}
Y_{\mathbf{i}\cdot j}(a_j\cdot \xi)
\ge  \frac{1}{2\alpha^2} \cdot 
Y_{\mathbf{i}\cdot j'}(a_{j'}\cdot \xi)\quad \mbox{for all $\xi \in I'$ and $(j, j')\in \mathcal{A}\times \mathcal{A}$.}
\end{equation}
In particular, setting $c=p_{\min}/(2\alpha^2)^2$, we have
\begin{equation}\label{eq:c}
p_j\cdot Y_{ \mathbf{i}\cdot j}(a_j\cdot \xi)^2\ge  c\cdot \sum_{j'\in  \mathcal{A}}
p_{j'}\cdot Y_{\mathbf{i}\cdot j'}(a_{j'}\cdot \xi)^2\quad \mbox{for all $\xi \in I'$ and $j\in \mathcal{A}$.}
\end{equation}

By the definition of $Y_{\mathbf{i}}(\xi)$ and the Schwartz inequality, we have
\begin{align*}
Y_{\mathbf{i}}(\xi)^2&\le 
\max\left\{\frac{1}{4} \left( \sum_{j\in \mathcal{A}}
p_j\cdot Y_{\mathbf{i}\cdot j}(a_{j}\cdot \xi)\right)^2, 
\left(\sum_{j\in \mathcal{A}}
p_j\cdot |X_{\mathbf{i}\cdot j}(a_{j}\cdot \xi)|\right)^2\right\}\\
&\le
\max\left\{\frac{1}{4} \sum_{j\in \mathcal{A}}
p_j\cdot Y_{\mathbf{i}\cdot j}(a_{j}\cdot \xi)^2, \;
\sum_{j\in \mathcal{A}}
p_j\cdot |X_{\mathbf{i}\cdot j}(a_{j}\cdot \xi)|^2\right\}\\
&\le
 \sum_{j\in \mathcal{A}}p_j\cdot 
 \max\left\{\frac{1}{4}
 Y_{\mathbf{i}\cdot j}(a_{j}\cdot \xi)^2, 
  |X_{\mathbf{i}\cdot j}(a_{j}\cdot \xi)|^2\right\}
\end{align*}
for all $\xi \in I'$. But, from the condition (II),  the last expression is bounded by 
\[
 \left(\sum_{j\in \mathcal{A}}p_j\cdot 
 Y_{\mathbf{i}\cdot j}(a_{j}\cdot \xi)^2\right) - \frac{1}{2} p_{j_0}\cdot Y_{ \mathbf{i}\cdot j_0}(a_{j_0}\cdot \xi)^2
\]
where $j_0\in \mathcal{A}$ is that in the condition (II). 
This together with the estimate (\ref{eq:c}) implies the conclusion of the lemma  (again with no exception for $1\le \nu\le N$),  choosing $\eta>0$ so small that  $e^{-\eta}\ge 1-(c/2)$. 

Below we consider the remaining case where neither of the conditions (I) and (II) holds. 
Notice that the estimates  (\ref{eq:Ydis}), (\ref{eq:y}) and (\ref{eq:c}) still hold  in this case, because these are consequences of Lemma \ref{lm:Lip} and  the assumption that the condition (I) does not hold. 
From the assumtion that the condition (II) does not hold,  for each $j\in \mathcal{A}$, there exists a point $\xi_j\in I$ such that 
\[
|X_{\mathbf{i}\cdot j}(a_j\cdot\xi_j)|\ge Y_{\mathbf{i}\cdot j}(a_j\cdot\xi_j)/2.
\]
From Lemma \ref{lm:Lip} and (\ref{eq:Ydis}), we have that
\begin{align*}
\left|\frac{d}{d\xi} (X_{\mathbf{i}\cdot j}(a_j\cdot \xi))\right|=\left|a_j\cdot\left(\frac{d}{d\xi} X_{\mathbf{i}\cdot j}\right)(a_j\cdot \xi)\right|&\le \frac{2\max_{i\in \mathcal{A}}|a_i|}{1-2\cdot \max_{i\in \mathcal{A}}|a_i|}\cdot Y_{ \mathbf{i}\cdot j}(a_j\cdot \xi)\\
&\le \frac{ 2\max_{i\in \mathcal{A}}|a_i|}{1-2\cdot \max_{i\in \mathcal{A}}|a_i|}\cdot \alpha \cdot Y_{\mathbf{i}\cdot j}(a_j\cdot \xi_j)
\end{align*}
for all $\xi \in I'$ and $j\in \mathcal{A}$. 
As we noted in the beginning, replacing the system $\{(S_i, p_i)\}_{i\in \mathcal{A}}$ by its iterates,  we may and do assume that the constant $\alpha$ in (\ref{eq:Ydis}) is  close to $1$ and that 
\[
\frac{2\max_{i\in \mathcal{A}}|a_i|}{1-2\cdot \max_{i\mathcal{A}}|a_i|}
\]
is close to $0$. Thus we may suppose without loss of generality  that the two inequalities above and (\ref{eq:Ydis}) imply  that we have
\begin{equation}\label{eq:XY}
|X_{\mathbf{i}\cdot j}(a_j\cdot \xi)|\ge Y_{\mathbf{i}\cdot j}(a_j\cdot\xi')/8
\end{equation}
and 
\begin{equation}\label{eq:XX}
\left|\frac{d}{d\xi}(X_{\mathbf{i}\cdot j}(a_j\cdot \xi))\right|\le 
\frac{ |X_{\mathbf{i}\cdot j}(a_j\cdot\xi')|}{10}
\end{equation}
for all $\xi, \xi' \in I'$  and $j\in \mathcal{A}$. 

Now we make use of the difference of the complex phases of the summands on the right hand side of (\ref{eq:relX}). 
From the definition of $Y_{\mathbf{i}}(\cdot)$ and  the relation (\ref{eq:relX}), we have that
\begin{align}
\left(\sum_{j \in \mathcal{A}} p_j \cdot Y_{\mathbf{i}\cdot j}(a_{j}\cdot \xi)\right)^2&-
|X_{\mathbf{i}}(\xi)|^2\ge 
\left(\sum_{j \in \mathcal{A}} p_j \cdot |T_j X_{\mathbf{i}\cdot j}(\xi)|\right)^2- 
\left|\sum_{j \in \mathcal{A}} p_j \cdot T_j X_{\mathbf{i}\cdot j}(\xi)\right|^2\notag\\
&\ge 
(1-\cos \Theta_{j',j''}(\xi))\cdot p_{j'} \cdot p_{j''}\cdot  |T_{j'}X_{\mathbf{i}\cdot j'}(\xi)|\cdot |
T_{j''} X_{\mathbf{i}\cdot j''}(\xi)|\label{eq:i-cos}
\end{align}
where $j'$ and $ j''$ are arbitrary two distinct elements in $\mathcal{A}$ and $\Theta_{j',j''}(\xi)$ denotes the difference between the complex arguments of $T_{j'}X_{\mathbf{i}\cdot j'}(\xi)$ and $T_{j''}X_{ \mathbf{i}\cdot j''}(\xi)$:
\[
\Theta_{j',j''}(\xi):=\arg\left(T_{j'}X_{\mathbf{i}\cdot j'}(\xi)\right)-\arg\left(T_{j'}X_{\mathbf{i}\cdot j''}(\xi)\right)\in \real/2\pi\mathbb{Z}.
\]
From (\ref{eq:y}) and (\ref{eq:XY}),  we have
\[ 
p_{j'}\cdot p_{j''}\cdot  |T_{j'}X_{ \mathbf{i}\cdot j'}(\xi)|\cdot |
T_{j''} X_{j''\cdot \mathbf{i}}(\xi)|\ge 
c'\cdot  \left(\sum_{j \in \mathcal{A}} p_j \cdot Y_{j\cdot \mathbf{i}}(a_{j}\cdot \xi)\right)^2,
\]
setting $c'=(1/8)^2\cdot (p_{\min}/(2\alpha^2))^{2}>0$. Putting this inequality in (\ref{eq:i-cos}), we obtain, for all $\xi \in I'$,  that
\begin{align*}
|X_{\mathbf{i}}(\xi)|^2&\le 
(1-c'(1-\cos\Theta_{j', j''}(\xi)))\cdot \left(\sum_{j \in \mathcal{A}} p_j \cdot Y_{ \mathbf{i}\cdot j}(a_{j}\cdot \xi)\right)^2
\intertext{and hence, from the definition of $Y_{\mathbf{i}}(\cdot)$, that }
|Y_{\mathbf{i}}(\xi)|^2&\le \max\{ 1/4, 1-c'(1-\cos\Theta_{j', j''}(\xi))\}\cdot \left(\sum_{j \in \mathcal{A}} p_j \cdot Y_{j\cdot \mathbf{i}}(a_{j}\cdot \xi)\right)^2\\
&\le(1-c'(1-\cos\Theta_{j', j''}(\xi)))\cdot \sum_{j \in \mathcal{A}} p_j \cdot Y_{ \mathbf{i}\cdot j}(a_{j}\cdot \xi)^2\quad \mbox{(as $c'<1/8$.)}
\end{align*}
Let us set $j'=1$ and $j''=m$ so that $b_{j'}-b_{j''}=1$ from the assumption (\ref{eq:b}). From  the definition of the operator $T_{j}$, we have that
\[
\frac{d}{d\xi} \Theta_{j',j''}(\xi)=-b_{j'}+b_{j''}+\frac{d}{d\xi} \arg(X_{j'\cdot \mathbf{i}}(a_{j'}\cdot\xi))-\frac{d}{d\xi} \arg(X_{ \mathbf{i}\cdot j''}(a_{j''}\cdot \xi)).
\]
Hence, from (\ref{eq:XX}), we get
\[
\left|\frac{d}{d\xi} \Theta_{j',j''}(\xi)-(b_{j''}-b_{j'})\right|=\left|\frac{d}{d\xi} \Theta_{j',j''}(\xi)-1\right|\le \frac{2}{10}.
\]
Since the length of the interval $I'$ and the subintervals $I'_\nu$ satisfy (\ref{eq:lengthI}), this implies that, for some small constant $\rho>0$ independent of $n$, we have
\[
\cos \Theta_{j',j''}(\xi) \le 1-\rho\quad\mbox{ on $ I'_{\nu}$}
\]
for all $1\le \nu \le N$ but for at most two exceptions.  We therefore obtain the conclusion of the lemma 
by choosing $\eta>0$ so small that 
$e^{-\eta}>1-c'\rho$. 
\end{proof}

\subsection{Proof of Theorem \ref{th:main}}
Let $n>0$ be a large integer. Let $\epsilon>0$ and 
$\eta>0$ be the constants  in Corollary \ref{cor} and Lemma \ref{lm:main} respectively. These constants do not depend on $n$.
For a sequence $\mathbf{i}\in \mathcal{A}^k$ with $0\le k<n$ and a point $\xi\in [-\check{r}^n, \check{r}^n]$, we set 
\[
r_{\mathbf{i}}(\xi)=
\begin{cases}
\eta,&\mbox{ if the inequality in (\ref{eq:mainest}) holds;}\\
0,&\mbox{otherwise}
\end{cases}
\]
for $\xi\in I(\mathbf{i})$. 
From (\ref{eq:YXn}) and the Schwartz inequality,  we have, in general,
\[
Y_{\mathbf{i}}(\xi)^2
\le  \sum_{j\in \mathcal{A}} p_j\cdot Y_{\mathbf{i}\cdot }(a_j\cdot \xi)^2\quad \mbox{for $\xi \in I(\mathbf{i})$.}
\]
So we have
\[
Y_{\mathbf{i}}(\xi)^2
\le  \exp(-r_{\mathbf{i}}(\xi))\cdot\sum_{j\in \mathcal{A}} p_j\cdot Y_{\mathbf{i}\cdot j}(a_j\cdot \xi)^2\quad \mbox{for $\xi \in I(\mathbf{i})$.}
\]
Applying this inequality inductively, we obtain, for $\xi \in I(\emptyset)=[-\check{r}^n, \check{r}^n]$, that
\begin{align*}
|\mathcal{F}\mu_n(\xi)|^2\le Y_{\emptyset}(\xi)^2&\le \sum_{\mathbf{i}\in \mathcal{A}^n} 
\exp\left( -\sum_{k=2}^{n+1} r_{[\mathbf{i}]_{k}^n}(a_{[\mathbf{i}]_{k}^n}\cdot\xi)\right)\cdot  p_{\mathbf{i}}\cdot Y_{\mathbf{i}}(a_{\mathbf{i}}\cdot \xi)^2\\
&=
\sum_{\mathbf{i}\in \mathcal{A}^n} 
\exp\left(- \sum_{k=2}^{n+1} r_{[\mathbf{i}]_{k}^n}(a_{[\mathbf{i}]_{k}^n}\cdot\xi)\right)\cdot  p_{\mathbf{i}}\qquad \mbox{by (\ref{eq:YXn})}
\end{align*}
where $[\,\mathbf{i}\,]_{k}^{n}\in\mathcal{A}^{n-k+1}$ is defined by 
\[
[\,\mathbf{i}\,]_{k}^{n}(i)=\mathbf{i}(k+i-1) \quad \mbox{for $1\le i\le n-k+1$.}
\]

To get the conclusion of Theorem \ref{th:main}, we estimate the Lebesgue measure of the subset 
\begin{equation}\label{eq:area}
Z_n:=\left\{ \xi\in  [-\check{r}^n, \check{r}^n]\;\left| \;
|\mathcal{F}\mu_n(\xi)|^2\ge 3\exp(-s \eta n)\right.\right\}
\end{equation}
for small $s>0$. 
Let $\mathcal{E}_n\subset  \mathcal{A}^n$ be the set of sequences $\mathbf{i}\in \mathcal{A}^n$ such that $|a_{\mathbf{i}}|^{-1}< 2\check{r}^n$. We will ignore the sequences in this subset. This does not cause problems because the probability of such set of sequences is small. In fact, by Lemma \ref{lm1}, we may suppose 
\[
\mathbf{p}^n(\mathcal{E}_n)\le  \exp(-s\eta n)
\]
by letting $s>0$ be small. For each $\xi \in Z_n$, we have
\[
 \sum_{\mathbf{i}\in \mathcal{A}^n\setminus \mathcal{E}_n} \exp\left(- \sum_{k=2}^{n+1} r_{[\mathbf{i}]_{k}^n}(a_{[\mathbf{i}]_{k}^n}\cdot\xi)\right)\cdot  p_{\mathbf{i}}\ge 2\exp(-s \eta n).
\]
So, if we get a bound for the integration of the function on the left hand side above, we also get a corresponding bound for the Lebesgue measure of $Z_n$. For more precise argument, we put 
\[
\mathbf{I}_{\mathbf{i}}=\int_{-\check{r}^n}^{\check{r}^n}
\max\left\{0, \exp\left(- \sum_{k=2}^{n+1} r_{[\mathbf{i}]_{k}^n}(a_{[\mathbf{i}]_{k}^n}\cdot\xi)\right)-\exp(-s\eta n)\right\}d\xi
\]
for each $\mathbf{i}\in \mathcal{A}^n$. Then 
have
\begin{equation}\label{eq:z}
\exp(-s \eta n)\cdot \mathrm{Leb} Z_n \le \sum_{\mathbf{i}\in \mathcal{A}^n\setminus \mathcal{E}_n} p_{\mathbf{i}}\cdot \mathbf{I}_{\mathbf{i}}.
\end{equation}

We are going to estimate the integral $\mathbf{I}_{\mathbf{i}}$ for $\mathbf{i}\in \mathcal{A}^n\setminus \mathcal{E}_n$ by an elementary  combinatorial argument. First of all,  note that,   if $\mathbf{i}\notin \mathcal{E}_n$, we have $|a_{\mathbf{i}}|^{-1}\ge 2\check{r}^n$ and the partition $\Xi_{\mathbf{i}}$ consists of the single interval $[-\check{r}^n, \check{r}^n]$.
From the construction of the partitions $\Xi_{\mathbf{i}}$, 
 each interval in the partition $\Xi_{[\mathbf{i}]_{k}^n}$ with $0\le k\le n-1$ consists of at most 
\[
\kappa:=\left[2\cdot \max_{i\in \mathcal{A}} \cdot |a_{i}|^{-1}\right].
\]
consecutive intervals in $\Xi_{[\mathbf{i}]_{k+1}^{n}}$.
And, from Lemma \ref{lm:main}, the condition $r_{[\mathbf{i}]_{k}^n}(a_{[\mathbf{i}]_{k}^n}\cdot\xi)=0$ holds for a point $\xi$ in at most two intervals among them.
Since the integrand in the definition of $\mathbf{I}_{\mathbf{i}}$ is bounded by $1$ and  takes a positive value only if 
\[
\sum_{k=2}^{n+1} r_{[\mathbf{i}]_{k}^n}(a_{[\mathbf{i}]_{k}^n}\cdot\xi)\le s \eta n
\]
The last condition 
implies that $r_{[\mathbf{i}]_{k}^n}(a_{[\mathbf{i}]_{k}^n}\cdot\xi)=0$ holds for all $2\le k\le  n+1$ 
 but for at most $[s n]$ exceptions. 
 Hence we obtain 
\[
\mathbf{I}_{\mathbf{i}}\le  \sum_{\nu\le [sn]} {n\choose \nu}\cdot  2^{n-\nu}\cdot \kappa^\nu.
 \]
By using Stirling's formula
\[
\log m!= m\log m -m +\mathcal{O}(\log m)
\]
and the inequality $\log(1+x)\le x$ that holds for any $x\ge 0$, 
we get
\begin{align*}
\log \left({n\choose \nu}\cdot  2^{n-\nu}\cdot \kappa^\nu\right)
&\le \log n!-\log \nu!-\log ((n-\nu)!)\\
&\qquad +(n-\nu)\log 2+\nu \log \kappa+\mathcal{O}(\log n)\\
&\le (n-\nu) \log \left(1+\frac{\nu}{n-\nu}\right)+\nu \log \left(1+\frac{n-\nu}{\nu}\right)\\
&\qquad +(n-\nu)\log 2+\nu \log \kappa+\mathcal{O}(\log n)\\
&\le \nu \log \kappa+n(1+\log 2)+\mathcal{O}(\log n).
\end{align*}
Hence  we obtain
\begin{align*}
\mathbf{I}_{\mathbf{i}}&\le \exp\bigg(n(s\log \kappa+2)+\mathcal{O}(\log n)\bigg).
\end{align*}
Using this estimate in (\ref{eq:z}), we conclude 
\[
\log  \mathrm{Leb} Z_n \le n(s (\eta+\log \kappa)+2)+\mathcal{O}(\log n),
\]
that is, 
\[
\frac{1}{n} \log \left(\mathrm{Leb}\{x\in [-\check{r}^n , \check{r}^n]\mid |\mathcal{F}\mu_n(\xi)| \ge 3e^{-s\eta n}  \}\right)
\le s(\eta+ \log \kappa)+2+\mathcal{O}\left(\frac{\log n}{n}\right).
\]
Together with Lemma \ref{lm:apx}, this implies
\[
R(c;\mu)=\varlimsup_{t\to \infty}\frac{1}{t}\log \left(\mathrm{Leb}\{x\in [-e^t , e^t]\mid |\mathcal{F}\mu(\xi)| \ge e^{-ct}  \}\right)\le 
\frac{s (\eta+\log \kappa)+2}{\log \check{r}}
\]
with $c=(s\cdot \eta)/(\log \check{r})$, provided $s>0$ is sufficiently small. This implies 
\[
\lim_{c\to +0}R(c;\mu)\le 
\frac{2}{\log \check{r}}.
\]
Since  we may assume $\check{r}$ to be  arbitrarily large by replacing the system $\{S_i, p_i\}_{i\in \mathcal{A}}$ by its iterates, we obtain Theorem \ref{th:main}.

\end{document}